\documentclass[11pt]  {amsart} 
\usepackage{amsmath,latexsym,amssymb,amsmath,
amscd,amsthm,amsxtra}

\newtheorem{definition}{Definition}[section]
\newtheorem{theorem}{Theorem}[section]

\newtheorem{lemma}{Lemma}[section]

\def\tr{\mathrm{tr}}

\def\l{\lambda}

\def\R{\mathbb R}

\def\<{\langle}
\def\>{\rangle}
\def\div{{\rm div}}

\def \ds{\displaystyle}
\def \vs{\vspace*{0.1cm}}

\begin{document}

\title[A sharp lower bound for the first eigenvalue]{A sharp lower bound for the first eigenvalue on Finsler manifolds}
\author{Guofang Wang}
\address{ Albert-Ludwigs-Universit\"at Freiburg,
Mathematisches Institut,
Eckerstr. 1,
79104 Freiburg, Germany}
\email{guofang.wang@math.uni-freiburg.de}
\thanks{GW is partly supported by SFB/TR71 ``Geometric partial differential equations'' of DFG. CX is supported by China Scholarship Council.}
\author{Chao Xia}\address{Albert-Ludwigs-Universit\"at Freiburg, Mathematisches Institut,  Eckerstr. 1, 79104 Freiburg, Germany}
\email{chao.xia@math.uni-freiburg.de}

\date{}


\begin{abstract}{ 
In this paper, we give a sharp lower bound for the first (nonzero) Neumann eigenvalue of Finsler-Laplacian in Finsler manifolds in terms of  diameter, dimension, weighted Ricci curvature.
}

\


\noindent This paper is  dedicated to Professor S. S. Chern  on the occasion of his 100th birthday


\end{abstract}

\medskip

\maketitle

\section{Introduction}

The study of  the first (nonzero) eigenvalue of Laplacian in Riemannian manifolds plays an important
 role
  in  differential geometry. The first result on this subject, due to Lichnerowicz \cite{Lic},
 says that for an $n$-dimensional smooth compact manifold without boundary,  the first eigenvalue 
$\lambda_1$ can be estimated below by $\frac{n}{n-1}K$, provided that its $Ric\geq K>0$. 
In this case, Obata \cite{O} established a rigidity result, asserting the optimality 
of Lichnerowicz' estimate. Namely, $\lambda_1=K$ if and only if $M$ is isometric to the $n$ 
dimensional sphere with constant curvature $\frac{1}{n-1}K$. When $K=0$, Li-Yau \cite{Li,LY} developed a method,
 which depends on the gradient estimate of the eigenfunctions, to give the lower bound of the first 
eigenvalue via diameter $d$, precisely, $\lambda_1\geq \frac{\pi^2}{2d^2}$. Their method had been improved by 
Zhong-Yang \cite{ZY} to obtain $\lambda_1\geq \frac{\pi^2}{d^2}$, which is optimal in the sense 
that equality can be attained for one dimensional circle. Very recently, 
Hang and Wang showed that $\l_1>\frac {\pi^2}{d^2}$ in \cite{HW}, if the dimension $n>1$.
 These results also hold true when $M$ is a manifold with convex boundary. 
When $M$ is a convex domain in $\R^n$, this is a classical result of Payne-Weinberger \cite{PW}. 
Later Chen-Wang \cite{CW} and
 Bakry-Qian \cite{BQ} combined these results into a same framework, and gave  
estimates for the first eigenvalue of  very general elliptic symmetric operators, 
via diameter and Ricci curvature. 
This sharp estimate on Riemannian manifolds
 has been also generalized to Alexandrov spaces by Qian-Zhang-Zhu \cite{QZZ}.

Finsler geometry attracts many attentions in recent years, since it has broader applications in nature science. Simultaneously Finsler manifold is one of the most natural metric measure spaces, which plays an important role in many aspects in mathematics.
There exists a natural Laplacian on Finsler manifolds, which we call here \textit{Finsler-Laplacian}. 
Unlike the usual Laplacian, the Finsler-Laplacian is a nonlinear operator. 
The objective  of this paper is to study the lower bound for the first (nonzero) eigenvalue of
 this Finsler-Laplacian on Finsler manifolds. In \cite{Oh} Ohta  
introduced the weighted Ricci curvature $Ric_N$ for $N\in [n,\infty]$ of Finsler manifolds,
following the work of Lott-Villani \cite{LV} and Sturm \cite{S1} on metric measure space.
He proved the equivalence of the lower lower boundedness of the $Ric_N$ and
 the curvature-dimension conditions $CD(K,N)$ in \cite{LV,S1}. 
As a byproduct, he obtained a Lichnerowicz type estimate on the first eigenvalue 
of Finsler-Laplacian under the assumption $Ric_N\geq K>0$.
Another interesting type of
 eigenvalue estimates was obtained by Ge-Shen in \cite{GS}, namely the Faber-Krahn type inequality for the first Dirichlet eigenvalue of the Finsler-Laplacian holds. 
See also  \cite{BFK} and \cite{WuX}.
Recently, we proved in \cite{WX}
 the Li-Yau-Zhong-Yang type sharp estimate for a so-called anisotropic Laplacian
on a Minkowski space, which could be viewed as
 the simplest, but interesting and important
 case of non-Riemannian Finsler manifolds.
In this paper, we shall generalize 
the results in \cite{WX} to general Finsler manifolds. Moreover, as in \cite{CW} and \cite{BQ}, 
we shall put the Li-Yau-Zhong-Yang type and the Lichnerowicz type sharp estimates into a uniform framework. 

\

Our main result of this paper is
\begin{theorem}\label{main thm} Let $(M^n,F,m)$ be a $n$-dimensional compact Finsler measure space,  equipped with a Finsler structure $F$ and a smooth measure $m$, without boundary or with a convex boundary. Assume that $Ric_N\geq K$ for some real numbers $N\in [n,+\infty]$ and $K\in \mathbb{R}$.
Let $\lambda_1$  be the first (nonzero) Neumann eigenvalue of the Finsler-Laplacian $\Delta_m$, i.e.,
\begin{equation}\label{first eigenvalue}
-\Delta_m u=\lambda_1 u, \hbox{ in } M,
\end{equation}
with a Neumann boundary condition
\begin{equation}\label{eq2}
 \nabla u(x)\in T_x(\partial M), 
\end{equation}
if $\partial M$ is not empty.
Then
\begin{eqnarray}
\lambda_1\geq \lambda_1(K,N,d),
\end{eqnarray}
where $d$ is the diameter of $M$, $\lambda_1(K,N,d)$ represents the first (nonzero) eigenvalue of the $1$-dimensional problem
\begin{eqnarray}
v''-T(t)v'=-\lambda_1(K,N,d) v \hbox{ in }(-\frac{d}{2},\frac{d}{2}),\quad v'(-\frac{d}{2})=v'(-\frac{d}{2})=0,
\end{eqnarray}
with $T(t)$ varying according to different values of $K$ and $N$. $T$ is explicitly defined by
\begin{equation}\label{T}
T(t)=\left\{
\begin{array}{lll} \ds\vs
\sqrt{(N-1)K}\tan\left(\sqrt{\frac{K}{N-1}}t\right), &\hbox{ for }K>0, 1<N<\infty,\\
\ds\vs -\sqrt{-(N-1)K}\tanh\left(\sqrt{-\frac{K}{N-1}}t\right), &\hbox{ for }K<0, 1<N<\infty,\\
0, &\vs \hbox{ for }K=0,1<N<\infty,\\
Kt, &\hbox{ for }N=\infty.
\end{array}\right.
\end{equation}
\end{theorem}


The precise definition of the Finsler measure space, convex boundary, diameter $d$, weighted Ricci curvature $Ric_N$, gradient vector field 
$\nabla$, Finsler-Laplacian $\Delta_m$ will be given in Section 2 below.

Equivalently, Theorem \ref{main thm} gives an optimal Poincar\'e inequality in Finsler manifolds.
\begin{theorem} Under the same assumptions as in Theorem \ref{main thm}, we have
\begin{eqnarray}
\int_M F^2(\nabla u) dm\geq  \lambda_1(K,N,d) \int_M (u-\bar u)^2 dm,\end{eqnarray}
where $\bar u$ is the average of $u$.
\end{theorem}

In the case of $K>0$ and $N=n$, Theorem \ref{main thm} sharpens the 
 Lichnerowicz type estimate given by Ohta \cite{Oh},
 since by Meyer Theorem, $d\leq \frac{\pi}{\sqrt{(n-1)K}}$,
 and  $\lambda_1(K,N,d)\geq \lambda_1(K,n,\frac{\pi}{\sqrt{(n-1)K}})= \frac{nK}{n-1}$. 
In the case of $K=0$ and $N=n$, Theorem \ref{main thm} gives
 the Li-Yau-Zhong-Yang type sharp estimate for the Finsler-Laplacian,
since $\lambda_1(0,n,d)=\frac{\pi^2}{d^2}.$ We remark that the Minkowski space $(\mathbb{R}^n,F)$ equipped with the $n$-dimensional Lebesgue measure 
satisfies that $Ric_N\geq K$ with $K=0$ and $N=n$ (see e.g. \cite{Vi} Theorem in page 908 and \cite{Oh}, Theorem 1.2), hence Theorem \ref{main thm} covers the estimate in \cite{WX}.

Our proof goes along the line of Bakry-Qian \cite{BQ}. The 
technique is based on a comparison theorem on the gradient of the first 
eigenfunction with that of  a one dimensional (1-D) model function (Theorem \ref{Comparison thm}), 
which was developed by Kr\"oger \cite{Kr} and improved by Chen-Wang \cite{CW} and Bakry-Qian \cite{BQ}. 
By using a Bochner-Weizenb\"ock formula established  recently by Ohta-Sturm \cite{OS2}, 
we find that the one dimensional model coincides with that in the Riemannian case, 
as presented in Theorem \ref{main thm}. 
It should be not so surprising,
 because when we consider $F$ in $\mathbb{R}$, 
it can only be two pieces of linear functions. 
Since the 1-D model has been extensively studied in \cite{BQ}, it also eases our situation, 
although we deal with a nonlinear operator. One  difficulty  arises when we  deal with the Neumann boundary problem, 
since the convexity of boundary could  not be directly applied
due to the difference between the metric induced from the boundary itself and the metric 
induced from the gradient of the first eigenfunction. We will establish
 some equivalence between them (see Lemma \ref{lem1} and \ref{lem2}) to overcome this difficulty.
Another ingredient is a comparison theorem on the maxima of eigenfunction with 
that of  the 1-D model function (Theorem \ref{Comparison thm2}). 
Everything in \cite{BQ} works except the boundedness of the Hessian of eigenfunctions around a 
critical point (since the eigenfunction is only $C^{1,\alpha}$ among $M$), which was used to prove (\ref{eq12}). 
Here we  avoid the use of the Hessian of eigenfunctions by using the comparison theorem on the gradient.
For the rest we  follow step by step the work of Bakry-Qian \cite{BQ} to get Theorem \ref{main thm}.

\

This paper is organized as follows. In Section 2,
the fundamentals in Finsler geometry is briefly introduced
 and the recent work of Ohta-Sturm is reviewed. 
We shall first prove the comparison theorem on the gradient and 
on the maxima of the eigenfunction and then Theorem \ref{main thm}  in Section 3.

\

\section{Preliminaries on Finsler geometry}

In this section  we briefly recall the fundamentals of Finsler geometry,
 as well as  the recent developments on the analysis of Finsler geometry by Ohta-Sturm \cite{Oh, OS1, OS2}.
For Finsler geometry, we refer to \cite{BCS} and \cite{Sh}.

\subsection{Finsler structure and Chern connection}

Let $M^n$ be a smooth, connected $n$-dimensional  manifold. A function $F:TM\to[0,\infty)$ is called a \textit{Finsler structure} if it satisfies the following properties:
\begin{itemize}
\item[(i)] $F$ is $C^\infty$ on $TM\setminus\{0\}$;
\item [(ii)] $F(x, tV) = t F(x, V)$ for all $(x, y)\in TM$ and all $t > 0$;
\item [(iii)] for every $(x, V)\in TM\setminus\{0\}$, the matrix
$$g_{ij}(V):=\frac{\partial^2}{\partial V_i\partial V_j}(\frac12 F^2)(x, V)$$ is positive definite.
\end{itemize}
Such a pair $(M^n,F)$ is called a \textit{Finsler manifold}. 
A Finsler structure is said to be \textit{reversible} if, in addition, $F$ is even. Otherwise $F$
is non-reversible.
By a {\it Finsler measure space}  we mean  a triple $(M^n,F,m)$ constituted with a smooth, connected $n$-dimensional manifold $M$, a Finsler structure $F$ on $M$ and a measure $m$ on $M$.

For $x_1,x_2\in M$, the \textit{distance function} from $x_1$ to $x_2$ is defined by
$$d(x_1,x_2):=\inf_\gamma \int_0^1 F(\dot{\gamma}(t))dt,$$
where the infimum is taken over all $C^1$-curves $\gamma:[0,1]\to M$ such that $\gamma(0)=x_1$ and $\gamma(1)=x_2$. Note that the distance function may not be symmetric unless $F$ is reversible.
A $C^\infty$-curve $\gamma:[0,1]\to M$ is called a \textit{geodesic} if $F(\dot{\gamma})$ is constant and it is locally minimizing.
The \textit{diameter} of $M$ is defined by $$d:=\sup_{x,y\in M} d(x,y).$$
The \textit{forward} and \textit{backward open balls} are defined by
$$B^+(x,r):=\{y\in M: d(x,y)<r\},\quad B^-(x,r):=\{y\in M: d(y,x)<r\}.$$ We denote
$B^\pm (x,r):=B^+(x,r)\bigcup B^-(x,r).$

For every non-vanishing vector field $V$, $g_{ij}(V)$ induces a Riemannian structure $g_V$ of $T_xM$  via
$$g_V(X,Y)=\sum_{i,j=1}^n g_{ij}(V)X^iY^j, \hbox{ for }X,Y\in T_xM.$$
In particular, $g_V(V,V)=F^2(V)$.

Let $\pi: TM\setminus\{0\}\to M$ the projection map. The pull-back bundle $\pi^*TM$ admits a unique linear connection, which is the \textit{Chern connection}. The Chern connection is determined by the following structure equations, which characterize ``torsion freeness":
\begin{eqnarray}\label{torsion free}
D_X^V Y-D_Y^V X=[X,Y]\end{eqnarray}
 and ``almost $g$-compatibility"
\begin{eqnarray}\label{compatible}
Z(g_V(X,Y))=g_V(D_Z^V X,Y)+g_V(X,D_Z^V Y)+C_V(D_Z^V V,X,Y)\end{eqnarray}
for $V\in TM\setminus\{0\}, X,Y,Z\in TM$. Here $$C_V(X,Y,Z):=C_{ijk}(V)X^i Y^j Z^k=\frac{1}{4}\frac{\partial^3 F^2}{\partial V^i V^j V^k}(V)X^i Y^j Z^k$$ denotes the \textit{Cartan tensor} and $D_X^V Y$ the \textit{covariant derivative} with respect to reference vector $V\in TM\setminus \{0\}$.
We mention here that $C_V(V,X,Y)=0$ due to the homogeneity of $F$.
In terms of the Chern connection, a geodesic $\gamma$ satisfies $D_{\dot{\gamma}}^{\dot{\gamma}} {\dot{\gamma}}=0$.
For local computations in Finsler geometry, we refer to \cite{Sh}.

\subsection{Hessian and Finsler-Laplacian}

We shall introduce the Finsler-Laplacian on Finsler manifolds. First of all, we recall the notion of the Legendre transform.

Given a Finsler structure $F$ on $M$, there is a natural dual norm $F^*$ on the cotangent space $T^*M$, which is defined by
$$F^*(x,\xi):=\sup_{F(x, V)\leq 1} \xi(V) \hbox{ for any }\xi\in T_x^*M.$$
One can show that  $F^*$ is also a Minkowski norm on $T^*M$ and
$$g_{ij}^*(\xi):=\frac{\partial^2}{\partial \xi_i\partial \xi_j}(\frac12 F^{*2})(x, \xi)$$ is positive definite for every $(x, \xi)\in T^*M\setminus\{0\}$.

The \textit{Legendre transform} is defined by the map $l:T_xM\to T_x^*M:$
\begin{equation*}l(V):=\left\{
\begin{array}{lll} g_V(V,\cdot)&\hbox{ for }V\in T_xM\setminus\{0\},\\
0 &\hbox{ for }V=0.\\
\end{array}
\right.
\end{equation*}
One can verify that $F(V)=F^*(l(V))$ for any $V\in TM$ and $g_{ij}^*(x,l(V))$ is the inverse matrix of $g_{ij}(x,V)$.

Let $u:M\to \mathbb{R}$ be a smooth function on $M$ and $Du$ be its differential $1$-form. The \textit{gradient} of $u$ is defined as $\nabla u(x):=l^{-1}(Du(x))\in T_xM$. Denote $M_u:=\{Du\neq 0\}$. Locally we can write in coordinates
\begin{equation*}
\nabla u=\sum_{i,j=1}^n g^{ij}(x,\nabla u)\frac{\partial u}{\partial x_i}\frac{\partial}{\partial x_j} \hbox{ in } M_u.
\end{equation*}

The \textit{Hessian} of $u$ is defined by using Chern connection as
\begin{eqnarray}\label{hess sym}
\nabla^2 u(X,Y)=g_{\nabla u}(D_X^{\nabla u} \nabla u, Y),
\end{eqnarray}
One can  show that $\nabla^2 u(X,Y)$ is symmetric, see \cite{WuX} and \cite{OS2}. Indeed, using (\ref{torsion free}) and (\ref{compatible}) and noticing that $C_{\nabla u}(\nabla u, X, Y)=0$, we have
\begin{eqnarray*}
g_{\nabla u}(D_X^{\nabla u} \nabla u, Y)
&=& X(g_{\nabla u}(\nabla u, Y))-g_{\nabla u}(\nabla u, D_X^{\nabla u} Y)\\
&=& XY(u)-g_{\nabla u}(\nabla u, D_Y^{\nabla u} X+[X,Y])\\
&=& YX(u)+[X,Y](u)-g_{\nabla u}(\nabla u, D_Y^{\nabla u} X)-[X,Y](u)\\
&=& Y(g_{\nabla u}(\nabla u, X))-g_{\nabla u}(\nabla u, D_Y^{\nabla u} X)
=g_{\nabla u}(D_Y^{\nabla u} \nabla u, X).
\end{eqnarray*}

In order to define a Laplacian on Finsler manifolds, we need a measure $m$ (or a volume form $dm$) on $M$. From now on, we consider the Finsler measure space $(M,F,m)$ equipped with a fixed smooth measure $m$.
Let $V\in TM$ be a smooth vector field on $M$. The \textit{divergence} of $V$ with respect to $m$ is defined by
\begin{eqnarray*}
\div_m V dm=d(V\lrcorner dm),
\end{eqnarray*}
where $V\lrcorner dm$ denotes the inner product of $V$ with the volume form $dm$.
In a local coordinate $(x^i)$, expressing $dm=e^\Phi dx^1 dx^2\cdots dx^n$, we can write $\div_m V$ as
\begin{eqnarray*}
\div_m V=\sum_{i=1}^n\left(\frac{\partial V^i}{\partial x^i}+V^i\frac{\partial \Phi}{\partial x^i}\right).
\end{eqnarray*}

A Laplacian, which is called the \textit{Finsler-Laplacian},  can now be defined by
\begin{eqnarray*}
\Delta_m u:=\div_m(\nabla u).
\end{eqnarray*}
We remark that  the Finsler-Laplacian is better to  be viewed  in a weak sense that for $u\in W^{1,2}(M)$,
\begin{eqnarray*}
\int_M \phi\Delta_m u dm=-\int_M D\phi(\nabla u) dm \hbox{ for }\phi\in C_c^\infty(M).
\end{eqnarray*}

The relationship between $\Delta_m u$ and $\nabla^2 u$ is that
\begin{eqnarray*}
\Delta_m u +D\Psi(\nabla u)= \tr_{g_{\nabla u}}(\nabla^2 u)=\sum_{i=1}^n \nabla^2 u(e_i,e_i),
\end{eqnarray*}
where $\Psi$ is defined by $dm=e^{-\Psi(V)}d\hbox{Vol}_{g_V}$ and $\{e_i\}$ is an orthonormal basis of $T_xM$ with respect to $g_{\nabla u}$. See e.g. \cite{WuX}, Lemma 3.3.

Given a vector field $V$, the \textit{weighted Laplacian} is defined on the weighted Riemannian manifold $(M,g_V,m)$ by
 \begin{eqnarray*}
\Delta_m^V u:=\div_m(\nabla^V u),
\end{eqnarray*}
where
\begin{equation*}\nabla^V u:=\left\{
\begin{array}{lll} \sum_{i,j=1}^n g^{ij}(x,V)\frac{\partial u}{\partial x_i}\frac{\partial}{\partial x_j}&\hbox{ for }V\in T_xM\setminus\{0\},\\
0 &\hbox{ for }V=0.\\
\end{array}
\right.
\end{equation*}

Similarly, the weighted Laplacian can be viewed in a weak sense that for $u\in W^{1,2}(M)$. We note that $\Delta_m^{\nabla u} u=\Delta_m u$.

\subsection{Finsler manifolds with boundary}

Assume that $(M,F,m)$ is a Finsler measure space with boundary $\partial M$, then we shall view $\partial M$ as a hypersurface embedded in $M$. $\partial M$ is also a Finsler manifold with a Finsler structure $F_{\partial M}$ induced by $F$.  For any $x\in\partial M$, there exists exactly two unit \textit{normal vectors} $\nu$, which are characterized by
$$T_x(\partial M)=\{V\in T_xM: g_\nu(\nu,V)=0, g_\nu(\nu,\nu)=1\}.$$ Throughout this paper, we choose the normal vector that points outward $M$. Note that, if $\nu$ is a normal vector, $-\nu$ may be not a normal vector unless $F$ is reversible.

The normal vector $\nu$ induces a volume form $dm_\nu$ on $\partial M$ from $dm$ by
$$V\lrcorner dm= g_\nu(\nu, V)dm_\nu, \quad\hbox{ for all } V\in T(\partial M).$$
One can check that Stokes theorem holds (see \cite{Sh}, Theorem 2.4.2) 
\begin{eqnarray*}
\int_M\div_m(V)dm=\int_{\partial M} g_\nu(\nu, V)dm_\nu.
\end{eqnarray*}

We recall the convexity of the boundary of $M$.

The \textit{normal curvature} $\Lambda_\nu(V)$ at $x\in \partial M$ in a direction $V\in T_x(\partial M)$ is defined by
\begin{eqnarray}
\Lambda_\nu(V)=g_\nu(\nu,D_{\dot{\gamma}}^{\dot{\gamma}} {\dot{\gamma}}(0)),
\end{eqnarray}
where $\gamma$ is the unique local geodesic for the Finsler structure $F_{\partial M}$ on $\partial M$ induced by $F$ with  the initial data $\gamma(0)=x$ and $\dot{\gamma}(0)=V$.

$M$ is said to has  \textit{convex boundary}  if for any $x\in\partial M$,
the normal curvature $\Lambda_\nu$ at $x$ is non-positive in any directions $V\in T_x(\partial M)$. We remark that the convexity of $M$ means that $D_{\dot{\gamma}}^{\dot{\gamma}} {\dot{\gamma}}(0)$ lies at the same side of $T_x M$ as $M$. Hence the choice of normal is not essential for the definition of convexity. (See Lemma 3.2 below).
There are several equivalent definitions of convexity, see for example \cite{BCGS} and \cite{Sh}.

\subsection{Weighted Ricci curvature}

The Ricci curvature of Finsler manifolds is defined as the trace of the flag curvature.
Explicitly, given two linearly independent vectors $V,W\in T_xM\setminus\{0\}$, the \textit{flag curvature} is defined by
$$\mathcal{K}^V(V,W)=\frac{g_V(R^V(V,W)W,V)}{g_V(V,V)g_V(W,W)-g_V(V,W)^2},$$
where $R^V$ is the \textit{Chern curvature} (or \textit{Riemannian curvature}):
$$R_V(X,Y)Z=D_X^V D_Y^V Z-D_Y^V D_X^V Z-D_{[X,Y]}^V Z.$$
Then the \textit{Ricci curvature} is defined by
$$Ric(V):=\sum_{i=1}^{n-1} \mathcal{K}(V,e_i),$$ where $e_1,\cdots,e_{n-1},\frac{V}{F(V)}$ form an orthonormal basis of $T_xM$ with respect to $g_V$.

We recall the definition of  the weighted Ricci curvature on Finsler manifolds, which was introduced by Ohta in \cite{Oh},
motivated by the work of Lott-Villani \cite{LV} and Sturm \cite{S1} on metric measure space.

\begin{definition}[\cite{Oh}]
Given a unit vector $V\in T_x M$, let $\eta: [-\varepsilon,\varepsilon]\to M$ be the
geodesic such that $\dot{\eta}(0)=V$. Decompose $m$ as $m=e^{-\Psi} d \hbox{vol}_{\dot{\eta}}$ along $\eta$, where $\hbox{vol}_{\dot{\eta}}$ is the volume form of $g_{\dot{\eta}}$ as a Riemannian metric.
Then
\begin{equation*}
\begin{array}{rcl} \ds\vs 
Ric_n(V)&:=&\left\{
\begin{array}{lll} \ds\vs Ric(V)+(\Psi\circ\eta)''(0) &\hbox{ if }(\Psi\circ \eta)'(0)=0,\\
-\infty &\hbox{ otherwise };
\end{array}\right. \\
\ds\vs Ric_N(V)&:=&Ric(V)+(\Psi\circ\eta)''(0)-\frac{(\Psi\circ\eta)'(0)^2}{N-n},\hbox{ for }N\in(n,\infty),
\\
\ds Ric_\infty(V)&:=& Ric(V)+(\Psi\circ\eta)''(0).
\end{array}
\end{equation*}
For $c\geq 0$ and $N\in [n,\infty]$,  define $$Ric_N(cV):=c^2Ric_N(V).$$
\end{definition}

Ohta proved in \cite{Oh} that, for $K\in\mathbb{R}$, the bound $Ric_N(V)\geq KF^2(V)$ is equivalent to Lott-Villani and Sturm's \textit{weak curvature-dimension condition} $CD(K,N)$.


\subsection{Bochner-Weitzenb\"ock formula}
The  following  Bochner-Weizenb\"ock type formula,  established by Ohta-Sturm in \cite{OS2},
plays an important role in this paper.
\begin{theorem}[\cite{OS2}, Theorem 3.6]\label{BW formula}
Given $u\in W_{loc}^{2,2}(M)\bigcap C^1(M)$ with $\Delta_m u\in W_{loc}^{1,2}(M)$, we have
\begin{eqnarray*} -\int_M D\eta\left(\nabla^{\nabla u}\left(\frac{F^2(x,\nabla u)}{2} \right)\right) dm=
\int_M \eta\bigg\{D(\Delta_m u)(\nabla u)+
 Ric_\infty(\nabla u)+\|\nabla^2 u\|^2_{HS(\nabla u)}\bigg\} dm\end{eqnarray*}
 as well as
\begin{eqnarray*} -\int_M D\eta\left(\nabla^{\nabla u}\left(\frac{F^2(x,\nabla u)}{2} \right)\right) dm\geq \int_M \eta\bigg\{D(\Delta_m u)(\nabla u)+  Ric_N(\nabla u)+\frac{(\Delta_m u)^2}{N}\bigg\} dm\end{eqnarray*}
for any $N\in [n,\infty]$ and all nonnegative functions $\eta\in W_c^{1,2}(M)\bigcap L^\infty(M)$. Here $\|\nabla^2 u\|^2_{HS(\nabla u)}$ denotes the Hilbert-Schmidt norm with respect to $g_{\nabla u}$.
\end{theorem}
Based on Bochner-Weitzenb\"ock formula, a similar argument as Bakry-Qian \cite{BQ} Theorem 6, leads to a refined inequality, which was referred to as
an extended curvature-dimension inequality there. Another direct proof was also given in \cite{WX}, Lemma 2.3.
\begin{theorem}Assume that $Ric_N\geq K$ for some $N\in [n,\infty]$ and some $K\in\mathbb{R}$. Given $u \in W_{loc}^{2,2}(M)\bigcap C^1(M)$ with $\Delta_m u\in W_{loc}^{1,2}(M)$, we have
\begin{eqnarray}\label{refined bochner} -\int_M D\eta\left(\nabla^{\nabla u}\left(\frac{F^2(x,\nabla u)}{2} \right)\right) dm&\geq &\int_M \eta\bigg\{D(\Delta_m u)(\nabla u) +K F(\nabla u)^2+\frac{(\Delta_m u)^2}{N}\nonumber\\&&+\frac{N}{N-1}\left(\frac{\Delta_m u}{N}-\frac{D(F^2(x,\nabla u))(\nabla u)}{2F^2(x,\nabla u)}\right)^2\bigg\} dm\end{eqnarray}
for any $N\in [n,\infty]$ and all nonnegative functions $\eta\in W_c^{1,2}(M)\bigcap L^\infty(M)$.
\end{theorem}

\section{Proof of Theorem \ref{main thm}}

We first remark that a weak eigenfunction $u\in W^{1,2}(M)$ of Finsler-Laplacian defined in (\ref{first eigenvalue}) has regularity that $u\in C^{1,\alpha}(M)\bigcap W^{2,2}(M)\bigcap C^\infty(M_u)$ (see \cite{GS}).

Let us  recall the 1-D models $L_{K,N}$ described in \cite{BQ}.
Let $K\in\mathbb{R}$ and $N\in (1,\infty]$.
\begin{itemize}
\item[(i)]For $K>0$ and $1<N<\infty$,
$L_{K,N}$ is defined on $\left(-\frac{\pi}{2\sqrt{K/(N-1)}},\frac{\pi}{2\sqrt{K/(N-1)}}\right)$ by
\begin{eqnarray*}
L_{K,N}(v)(t)=v''-\sqrt{K(N-1)}\tan(\sqrt{\frac{K}{N-1}} t)v';
\end{eqnarray*}
\item[(ii)]For $K<0$ and $1<N<\infty$,
$L_{K,N}$ is defined on $(-\infty,0)\bigcup(0,\infty)$ by
\begin{eqnarray*}
L_{K,N}(v)(t)=v''-\sqrt{-K(N-1)}\coth(\sqrt{-\frac{K}{N-1}} t)v'
\end{eqnarray*}
and on $(-\infty,\infty)$ by
\begin{eqnarray*}
L_{K,N}(v)(t)=v''-\sqrt{-K(N-1)}\tanh(\sqrt{-\frac{K}{N-1}} t)v';
\end{eqnarray*}
\item[(iii)]For $K=0$ and $1<N<\infty$,
$L_{K,N}$ is defined on $(-\infty,0)\bigcup(0,\infty)$ by
\begin{eqnarray*}
L_{K,N}(v)(t)=v''+\frac{N-1}{t}v'
\end{eqnarray*}
and on $(-\infty,\infty)$ by
\begin{eqnarray*}
L_{K,N}(v)(t)=v'';
\end{eqnarray*}
\item[(iv)]For $K\neq 0$ and $N=\infty$,
$L_{K,N}$ is defined on $(-\infty,\infty)$ by
\begin{eqnarray*}
L_{K,N}(v)(t)=v''-Ktv';
\end{eqnarray*}
\item[(v)]For $K= 0$ and $N=\infty$,
$L_{K,N}$ is defined on $(-\infty,\infty)$ by
\begin{eqnarray*}
L_{K,N}(v)(t)=v''-cv'
\end{eqnarray*}
for any constant $c$.
\end{itemize}
For convenience, we write $L_{K,N}(v)(t)=v''-T(t)v'$. It is easy to check that $T'=K+\frac{T^2}{N-1}.$
Denote by $\mu_{K,N}$ the invariant measure associated with $L_{K,N}$, that is, a measure satisfying
$\int_a^b L_{K,N}(v) d\mu_{K,N}=0$ for $v'(a)=v'(b)=0$. For instance, in the case (i), $d\mu_{K,N}=\cos^{N-1}(\sqrt{\frac{K}{N-1}}t)dt.$

\

The following gradient comparison theorem plays the most crucial role in the proof of our main theorem.

\begin{theorem} \label{Comparison thm}
Let $(M,F,m)$ and $\lambda_1$  be as in Theorem \ref{main thm}
and $u$ be the eigenfunction. Let $v$ be a solution of the 1-D model problem on some interval $(a,b)$:
\begin{equation}\label{1-D model}
L_{K,N}(v)=-\lambda_1 v,\quad v'(a)=v'(b)=0,\quad v'>0.
\end{equation}
Assume that $[\min u,\max u]\subset [\min v,\max v]$,
then
\begin{equation}\label{grad comparison}
F(x,\nabla u(x))\leq v'( v^{-1}(u(x))).
\end{equation}
\end{theorem}

\begin{proof} First, since $\int_M u=0$, $\min u<0$ while $\max u>0$. 
We may assume that $[\min u,\max u]\subset (\min v,\max v)$ by multiplying $u$ by a constant $0<c<1$. 
If we prove the result for this $u$, then letting $c\to 1$ implies the original statement.

Under the condition $[\min u,\max u]\subset (\min v,\max v)$, $v^{-1}$ is smooth on a neighborhood $U$ of $[\min u,\max u]$.

Consider  $P(x)=\psi(u)(\frac{1}{2}F^2(x,\nabla u)-\phi(u)),$ where $\psi,\phi\in C^\infty (U)$ are two positive smooth functions to be determined later.
We first  consider the case that $P$ attains its maximum at $x_0\in M$,
then study the case that $x_0\in\partial M$ if $\partial M$ is not empty.

\

\noindent{\bf Case 1.} $P$ attains its maximum at $x_0\in M$.

Due to the lack of regularity of $u$, we shall compute in the distributional sense. Let $\eta$ be any nonnegative function  in $W_c^{1,2}(M)\bigcap L^\infty(M)$. We first compute $-\int_M D\eta (\nabla^{\nabla u} P) dm$.
\begin{eqnarray*}
-\int_M D\eta (\nabla^{\nabla u} P) dm&=& -\int_M (\frac{\psi'}{\psi}P-\psi\phi')D\eta(\nabla u)+\psi D\eta(\nabla^{\nabla u}(\frac12 F^2(x,\nabla u))) dm\\
&=&\int_M -D[(\frac{\psi'}{\psi}P-\psi\phi')\eta](\nabla u)+\eta D(\frac{\psi'}{\psi}P-\psi\phi')(\nabla u)\\
&& - D(\psi\eta)(\nabla^{\nabla u}(\frac12 F^2(x,\nabla u))+\eta D\psi(\nabla^{\nabla u}(\frac12 F^2(x,\nabla u)) dm\\
&:=& I+II+III+IV.
\end{eqnarray*}

By using $Du(\nabla u)=F^2(x,\nabla u)=2(\frac{P}{\psi}+\phi)$ and $\Delta_m u=-\lambda_1 u$ in weak sense, we compute 
\begin{eqnarray*}
I=\int_M -\lambda_1 u(\frac{\psi'}{\psi}P-\psi\phi')\eta dm,
\end{eqnarray*}
\begin{eqnarray*}
II&=&\int_M \eta\big[((\frac{\psi''}{\psi}-\frac{\psi'^2}{\psi^2})P-\psi\phi''-\psi'\phi')Du(\nabla u)+\frac{\psi'}{\psi}DP(\nabla u)\big] dm\\
&=&  \int_M \eta\big[2((\frac{\psi''}{\psi}-\frac{\psi'^2}{\psi^2})P-\psi\phi''-\psi'\phi')(\frac{P}{\psi}+\phi)+\frac{\psi'}{\psi}DP(\nabla u)\big] dm,
\end{eqnarray*}
\begin{eqnarray*}
IV &=&\int_M \eta\psi'\big[\frac{1}{\psi}Du(\nabla^{\nabla u}P)+(-\frac{\psi'}{\psi^2}P+\phi')Du(\nabla u)\big] dm\\
&=&\int_M \big\{2\eta\psi'(-\frac{\psi'}{\psi^2}P+\phi')(\frac{P}{\psi}+\phi)+\hbox{ terms of }DP(\nabla u)\big\} dm.
\end{eqnarray*}

For the term III, we apply the refined integral Bochner-Weizenb\"ock formula (\ref{refined bochner}) to derive
\begin{eqnarray*}
III&\geq &\int_M \psi\eta\bigg[D(\Delta_m u)(\nabla u)+KF^2+\frac{(\Delta_m u)^2}{N}\\
&&+\frac{N}{N-1}\left(\frac{\Delta_m u}{N}-\frac{D(F^2(x,\nabla u))(\nabla u)}{2F^2(x,\nabla u)}\right)^2\bigg] dm\\
&=& \int_M \psi\eta\bigg[2(K-\lambda_1)(\frac{P}{\psi}+\phi)+\frac{\lambda_1^2 u^2}{N}\\
&&+\frac{N}{N-1}\left(\frac{-\lambda_1 u}{N}-(-\frac{\psi'}{\psi^2}P+\phi')-\frac{1}{\psi F^2}DP(\nabla u))\right)^2 \bigg] dm\\
&=&\int_M \psi\eta\bigg[2(K-\lambda_1)(\frac{P}{\psi}+\phi)+\frac{\lambda_1^2 u^2}{N-1}+\frac{N}{N-1}(-\frac{\psi'}{\psi^2}P+\phi')^2\\
&&+ \frac{2}{N-1}\lambda_1 u(-\frac{\psi'}{\psi^2}P+\phi') +\hbox{ terms of }DP(\nabla u)\bigg] dm .
\end{eqnarray*}

Combining all we obtain 
\begin{eqnarray}
-\int_M D\eta (\nabla^{\nabla u} P) dm&\geq &\int_M \eta\bigg\{\frac{1}{\psi}\left[2\frac{\psi''}{\psi}-(4-\frac{N}{N-1})\frac{\psi'^2}{\psi^2}\right]P^2\nonumber\\
&&+\left[2\phi\left(\frac{\psi''}{\psi}-2\frac{\psi'^2}{\psi^2}\right)-\frac{N+1}{N-1}\frac{\psi'}{\psi}\lambda u-\frac{2N}{N-1}\frac{\psi'}{\psi}\phi'+2(K-\lambda_1)-2\phi''\right]P\nonumber\\
&&+\psi\left[\frac{1}{N-1}\lambda_1^2u^2+\frac{N+1}{N-1}\lambda_1u \phi'+\frac{N}{N-1}\phi'^2+2(K-\lambda_1)\phi-2\phi\phi''\right]\nonumber\\
&&+\hbox{ terms of }DP(\nabla u)\bigg\} dm\nonumber\\
&:=&-\int_M \big\{a_1 P^2+a_2P+a_3+\hbox{ terms of }DP(\nabla u)\big\} dm.
\end{eqnarray}
Therefore, 
\begin{eqnarray}\label{eq000}
\Delta_m^{\nabla u} P+\hbox{ terms of }DP(\nabla u)=a_1 P^2+a_2P+a_3
\end{eqnarray}
holds in the distributional sense in $M$.

We claim that at the maximum point $x_0$ of $P$, 
\begin{eqnarray}\label{eq111}a_1 P^2+a_2P+a_3\leq 0.\end{eqnarray} In fact, if not, then in a neighborhood $\mathcal{U}$ of $x_0$, $a_1 P^2+a_2P+a_3>0$. It follows from (\ref{eq000}) that
the function $P$ is a strict subsolution to an elliptic operator in $\mathcal{U}$. By maximum principle, $P(x_0)<\max_{\partial \mathcal{U}} P$, which contradicts the maximality of $P(x_0)$.

It is interesting to see that the coefficients $a_i$, $i=1,2,3,$ coincide 
with that appeared in the Riemannian case (see e.g. \cite{BQ}, Lemma 1).
The next  step is to choose suitable positive functions $\psi$ and $\phi$ such that $a_1,a_2>0$ everywhere and $a_3=0$,
which has already been done in \cite{BQ}.
For completeness, we sketch the main idea here.

Choose $\phi(u)=\frac{1}{2}v'(v^{-1}(u))^2$, where $v$ is a solution of 1-D problem (\ref{1-D model}).
One can compute that $$\phi'(u)=v''(v^{-1}(u)),\phi''(u)=\frac{v'''}{v'}(v^{-1}(u)).$$
Set $t=v^{-1}(u)$ and $u=v(t)$ then
\begin{eqnarray*}\frac{a_3(t)}{\psi}&=&\frac{1}{N-1}\lambda_1^2v^2+\frac{N+1}{N-1}\lambda_1 v v''+\frac{N}{N-1}v''^2+(K-\lambda_1)v'^2-v'v'''
\\&=&-v'(v''-Tv'+\lambda_1v)'+\frac{1}{N-1}(v''-Tv'+\lambda_1v)(Nv''+Tv'+\lambda_1v)=0.\end{eqnarray*}
Here we have used that $T$ satisfies $T'=K+\frac{T^2}{N-1}.$
For $a_1,a_2$, we introduce $$X(t)=\lambda_1\frac{v(t)}{v'(t)},\quad\psi (u) =\exp(\int h(v(t))), \quad f(t)=-h(v(t))v'(t).$$
With these notations,
we have
\begin{eqnarray*}
f'=-h'v'^2+f(T-X),\end{eqnarray*}
\begin{eqnarray*}
v'|_{v^{-1}}^2 a_1\psi=2f(T-X)-\frac{N-2}{N-1}f^2-2f':=2(Q_1(f)-f'),\end{eqnarray*}
\begin{eqnarray*}
 a_2=f(\frac{3N-1}{N-1}T-2X)-2T(\frac{N}{N-1}T-X)-f^2-f':=Q_2(f)-f'.\end{eqnarray*}
 We may now use Corollary 3 in \cite{BQ}, which says that there exists a bounded function $f$ on $[\min u,\max u]\subset (\min v,\max v)$ such that $f'<\min\{Q_1(f),Q_2(f)\}$.

In view of (\ref{eq111}), we know that by our choice of $\psi$ and $\phi$,
 $P(x_0)\leq 0$, and hence $P(x)\leq 0$ for every $x\in M$, which leads to (\ref{grad comparison}).

\

\noindent{\bf Case 2.}   $\partial M\neq \emptyset$ and $x_0\in\partial M$.

To handle this case, we need to define a new normal vector field on $\partial M$, that is normal with respect to the Riemannian metric $g_{\nabla u}$. To be more general, for every $X\in TM$, there is a unique normal vector field $\nu_X$ such  that
\begin{eqnarray}\label{new normal}
g_X(\nu_X, Y)=0\hbox{ for any }Y\in T(\partial M),\quad g_X(\nu_X,\nu_X)=1,\quad g_\nu(\nu,\nu_X)>0.
\end{eqnarray}
A simple calculation shows that \begin{eqnarray}
g_X(\nu,\nu_X)>0.
\end{eqnarray}
Indeed, let $\nu_X=Z+a\nu$ for some $a\in\mathbb{R}$ and $Z\in T(\partial M)$. (\ref{new normal}) tells that $a>0$. Hence $g_X(\nu,\nu_X)=g_X(\frac{1}{a}(\nu_X-Z),\nu_X)=\frac{1}{a}>0$.

The following Lemma follows directly from the definition of $\nu$ and $\nu_X$.
\begin{lemma}\label{lem1}Let $X,Y\in TM$. Then $$g_\nu(\nu,Y)=0\Leftrightarrow Y\in T(\partial M)\Leftrightarrow g_X(\nu_X,Y)=0.$$
\end{lemma}\qed

Define four sets $$T^\nu_\pm M:=\{Y\in TM: g_\nu(\nu,Y)>0(<0)\}$$ and $$T^{\nu_X}_\pm M:=\{Y\in TM: g_X(\nu_X,Y)>0(<0)\}.$$

We have the following simple but important observation, which may be familiar to expects.
\begin{lemma}\label{lem2}
$T^\nu_+M=T^{\nu_X}_+M,\quad T^\nu_-M=T^{\nu_X}_-M$.\end{lemma}

\begin{proof} We first claim that either $T^\nu_+M\subset T^{\nu_X}_+M$ or $T^\nu_+M\subset T^{\nu_X}_-M$. 
Otherwise, there are two vector fields $Y_1,Y_2\in T^\nu_+M$, such that $g_X(\nu_X,Y_1)>0$ and $g_X(\nu_X,Y_2)<0$.
Then by the continuity of $g_X(\nu_X,\cdot)$ in $T^\nu_+M$, there exists $\overline{Y}\in  T^\nu_+M$ with $g_X(\nu_X,\overline{Y})=0$, which means $g_\nu(\nu, \overline{Y})=0$ from Lemma \ref{lem1}. A contradiction.
Taking into consideration that $\nu\in  T^{\nu_X}_+M$, we see that $T^\nu_+M\subset T^{\nu_X}_+M$. 
A similar argument implies that $T^{\nu_X}_+M\subset T^{\nu}_+M$. 
The second equivalence follows in a similar way.\end{proof}




Return to the case when $P$ attains its maximum at $x_0\in\partial M$. If $\nabla u(x_0)=0$, nothing needs to be proved. Thus we assume $x_0\in M_u$. Recall that $P\in C^\infty(M_u)$.
Since $\nu_{\nabla u}$ points outward due to its definition, by taking normal derivative of $P$ with respect to $\nu_{\nabla u}$, we have $$D P(\nu_{\nabla u})(x_0)\geq 0.$$

On one hand, the Neumann boundary condition $\nabla u\in T(\partial M)$ implies that $$g_{\nabla u}(\nu_{\nabla u},\nabla u)(x)=0,$$ or equivalently,
$$Du(\nu_{\nabla u})(x)=0\hbox{ for }x\in \partial M.$$
Thus we have \begin{eqnarray}\label{normal derivative} D P(\nu_{\nabla u})(x_0)=\frac12\psi(u)(D(F^2(\nabla u))(\nu_{\nabla u}))(x_0).\end{eqnarray}

On the other hand, using (\ref{compatible}) and the symmetry of $\nabla^2 u$, we have
\begin{eqnarray}\label{eq00}
 D(F^2(\nabla u))(\nu_{\nabla u}) &=&  D(g_{\nabla u}(\nabla u,\nabla u))(\nu_{\nabla u})\nonumber\\
&=&2 g_{\nabla u}(D_{\nu_{\nabla u}}^{\nabla u} (\nabla u), \nabla u)
= 2g_{\nabla u}(D_{\nabla u}^{\nabla u} (\nabla u), \nu_{\nabla u}).
\end{eqnarray}

By the convexity of $\partial M$, for any $X\in T(\partial M)$, $g_\nu(\nu, D_X^X X)\leq 0$.
 In particular, set $X=\nabla u$, we know that
\begin{eqnarray}\label{convexity}
g_\nu(\nu, D_{\nabla u}^{\nabla u} (\nabla u))\leq 0.
\end{eqnarray}
It follows from Lemma \ref{lem1} and \ref{lem2} that (\ref{convexity}) is equivalent to
\begin{eqnarray}\label{convexity1}
g_{\nabla u}(\nu_{\nabla u}, D_{\nabla u}^{\nabla u} (\nabla u))\leq 0.
\end{eqnarray}

Combining (\ref{normal derivative}), (\ref{eq00}) and (\ref{convexity1}), 
we conclude that $D P(\nu_{\nabla u})(x_0)\le 0$, and hence $D P(\nu_{\nabla u})(x_0)=0$.
The tangent derivatives of $P$ obviously vanish due to its maximality. Hence we have also 
$$\nabla P(x_0)=0.$$ 
Thus the proof for Case 1  works  in this case. This finishes the proof of Theorem \ref{Comparison thm}.
\end{proof}

Another ingredient is a comparison theorem for the maxima of the eigenfunctions.

\begin{theorem}\label{Comparison thm2}

Let $(M,F,m),\lambda_1$  be as in Theorem \ref{main thm} and $1<N<\infty$.
Let $v=v_{K,N}$ be a solution of the 1-D model problem on some interval $(a,b)$
$L_{K,N}v=-\lambda_1 v$, with initial data $v(a)=-1,v'(a)=0$, where
\begin{equation*}
a=\left\{
\begin{array}{lll}
-\frac{\pi}{2\sqrt{K/(N-1)}} &\hbox{ for }K>0,\\
0 &\hbox{ for }K\leq 0
\end{array}\right.
\end{equation*}
and $b=b(a)$ be the first number after $a$ with $v'(b)=0$. Denote $m_{K,N}=v_{K,N}(b)=\max(v)$.
Assume that $\lambda_1>\max\{\frac{KN}{N-1},0\}$ and $\min(u)=-1$.
 Then $\max u\geq m_{K,N}.$
\end{theorem}

\begin{proof}
We argue by contradiction. Suppose $\max(u)<m_{K,N}$. Then $[\min u,\max u]\subset [\min v,\max v]$.
The condition $\lambda_1>\max\{\frac{KN}{N-1},0\}$ ensures that \begin{equation*}
b\leq\left\{
\begin{array}{lll}
\frac{\pi}{2\sqrt{K/(N-1)}} &\hbox{ for }K>0,\\
\infty &\hbox{ for }K\leq 0,
\end{array}\right.
\end{equation*} which in turn ensures that $v'>0$ in $(a,b)$. Hence we could apply Theorem \ref{Comparison thm} for $u$ and $v$.

The same argument as Theorem 12 in \cite{BQ} implies that the ratio
\begin{eqnarray*}
R(c)=\frac{\int_{\{u\leq c\}} u dm}{\int_{\{v\leq c\}} v d\mu_{K,N}}
\end{eqnarray*}
is increasing on $[\min(u),0]$ and decreasing on $[0,\max(u)]$.
Therefore, for $c\leq -\frac12$, we have that
\begin{eqnarray}\label{eq10}
\quad m(\{u\leq c\})\leq 2\int_{\{u\leq c\}}|u| dm\leq 2R(0)\int_{\{v\leq c\}} |v| d\mu_{K,N}\leq 2R(0)\mu_{K,N}(\{v\leq c\}).
\end{eqnarray}
Let $c=-1+\varepsilon$ for $\varepsilon>0$ small.
A simple calculation gives that $v''(a)=\frac{\lambda_1}{N}$. Hence for $t$ close to $a$, $v''(t)$ has positive lower and upper bound.
Together with $v'(a)=0$, we see that $v(t)-v(a)\geq C(t-a)^2$. Thus if $t\in\{v\leq -1+\varepsilon\}$,  then $t\in (a,a+C\varepsilon^\frac12)$. It follows that
\begin{eqnarray}\label{eq11}\mu_{K,N}(\{v\leq -1+\varepsilon\})\leq \mu_{K,N}((a,a+C\varepsilon^\frac12)) \leq C\varepsilon^{N/2}.\end{eqnarray}

On the other hand, we shall prove that
\begin{eqnarray}\label{eq12}m(\{u\leq -1+\varepsilon\})\geq m(B^\pm(x_0,C\varepsilon^\frac12)).\end{eqnarray}
Let $x_0\in M$ be such that $u(x_0)=-1$. For any $x\in B^\pm(x_0,\delta)$ with $\delta$ small,  $u(x)$ is close to $-1$ and $s:=v^{-1}(u(x))$ is close to $a$. Thus we see again from the upper bound of $v''$ and $v'(a)=0$ that $v'(s)\leq C(s-a)$. Therefore, we  have from Theorem \ref{Comparison thm} that $F(x,\nabla u(x))\leq v'(v^{-1}(u(x)))\leq C(s-a)$ and
$F(x,\nabla v^{-1}(u(x)))=(v^{-1})'(u(x))F(x,\nabla u(x))\leq 1.$
In turn, we get
$$s-a=v^{-1}(u(x))-v^{-1}(u(x_0))\leq F(\tilde{x},\nabla v^{-1}(u(\tilde{x})))\delta\leq \delta,$$
and
$$u(x)\leq u(x_0)+F(\tilde{\tilde{x}},\nabla u(\tilde{\tilde{x}})) \delta \leq -1+C(s-a)\delta \leq -1+ C\delta^2,$$
for some $\tilde{x},\tilde{\tilde{x}}\in B^\pm(x_0,\delta)$.
Let $\varepsilon=C\delta^2$, we conclude $B^\pm(x_0,\delta)\subset \{u\leq -1+\varepsilon\}$, which implies (\ref{eq12}).

Combining (\ref{eq10}), (\ref{eq11}) and (\ref{eq12}), we see that there exists some constant $C>0$ such that
\begin{eqnarray}\label{eq13} m(B^\pm(x_0,r))\leq Cr^N.\end{eqnarray}
This will lead to a contradiction. In fact,
since $\max(u)<m_{K,N}$ and $m_{K,N}$ is continuous with respect to $(K,N)$, we also have that $\max(u)<m_{K,N'}$ for any $N'>N$ close to $N$. Argued as before, we will obtain (\ref{eq13}) with $N'$ instead of $N$, i.e.
\begin{eqnarray}\label{eq14} m(B^\pm(x_0,r))\leq Cr^{N'}.\end{eqnarray} However, the volume comparison theorem  for Finsler manifolds under the assumption of lower bound for $Ric_N$ (see \cite{Oh}, Th. 7.3), implies that
$m(B^\pm(x_0,r))\geq Cr^N$ for $r>0$ small. A contradiction  to (\ref{eq14}).
The previous argument also works in the case $x_0\in \partial M$. The proof is completed.
\end{proof}

Besides the comparison theorem on the gradient and maxima, in order to prove Theorem \ref{main thm}, we also need some properties of the 1-D models, which has been extensively studied in \cite{BQ}. We refer to \cite{BQ} for the elementary properties, meanwhile we list two of them, one presents the full range of the maximum function $m_{K,N}$, the other reveals that the central interval has the lowest first Neumann eigenvalue.

\begin{lemma}[\cite{BQ}, Section 3]\label{lem1 BQ}
Assume $1<N<\infty$ ($N=\infty$ resp.) and fix $\lambda>\max\{\frac{KN}{N-1},0\}$. Let $v, m$ be as in Theorem \ref{Comparison thm2}. Then for any $k\in [m,\frac1m]$ ($(0,\infty)$, resp.), there exists an interval which has the first Neumann eigenvalue $\lambda$ and a corresponding eigenfunction $\tilde{v}$ such that $\min\tilde{v}=-1,\max\tilde{v}=k.$
\end{lemma}

\begin{lemma}[\cite{BQ}, Th. 13]\label{lem2 BQ}
Let $\lambda_1(K,N,a,b)$ denotes the first Neumann eigenvalue of $L_{K,N}$ on the interval $(a,b)$. Then $\lambda_1(K,N,a,b)\geq \lambda_1(K,N,-\frac{b-a}{2},\frac{b-a}{2})=\lambda_1(K,N,b-a)$.
\end{lemma}

We now in a position to prove Theorem \ref{main thm}.

\noindent\textit{Proof of Theorem \ref{main thm}:}
Without loss of generality, we may assume that $\min u=-1$ and $0<\max u:=k\leq 1$. It was shown by Ohta \cite{Oh}, Cor. 8.5 that $\lambda_1\geq \frac{NK}{N-1}$ in the case of $K>0$. Choose $\tilde{K}<K$ close to $K$, we have $\lambda_1>\max\{\frac{\tilde{K}N}{N-1},0\}$.
Therefore, Theorem \ref{Comparison thm2} and Lemma \ref{lem1 BQ} imply that there exists an interval $[a,b]$ which has the first Neumann eigenvalue $\lambda_1$ and a corresponding eigenfunction $v$ such that $\min v=-1=\min u,\max v=\max u=k.$
Choose $x_1,x_2\in M$ with $u(x_1)=\min u,u(x_2)=k$ and
$\gamma(t):[0,1]\to M$ the minimal geodesic from $x_1$ to $x_2$. Consider the subset $I$ of [0,1] such that $\frac{d}{dt}u(\gamma(t))\geq 0$. By using Theorem \ref{Comparison thm}, we have
\begin{eqnarray*}d&\geq &  \int_{0}^1 F(\dot{\gamma}(t))dt\geq \int_I F(\dot{\gamma}(t))dt \\&\geq &\int_{0}^1 \frac{1}{F^*(D u)}Du(\dot{\gamma}(t))
dt=\int_{-1}^k \frac{1}{F(\nabla u)}du\\&\geq &\int_{-1}^k \frac{1}{v'(v^{-1}(u))}du=\int_{a}^{b} dt
=b-a.\end{eqnarray*}
A general property says that $\lambda_1 (\tilde{K},N,d)$ is monotone decreasing with respect to $d$. Hence $\lambda_1 (\tilde{K},N,b-a)\geq \lambda_1 (\tilde{K},N,d)$.
Finally, It follows from Lemma \ref{lem2 BQ} that $$\lambda_1\geq \lambda_1 (\tilde{K},N,b-a)\geq\lambda_1 (\tilde{K},N,d).$$ By letting $\tilde{K}\to K$, we get the conclusion $\lambda_1\geq\lambda_1 (K,N,d)$. \qed

\

\end{document}